\documentclass[12pt,twoside]{amsart}
\usepackage{amsmath, amsthm, amscd, amsfonts, amssymb, graphicx}
\usepackage[bookmarksnumbered, plainpages]{hyperref}

\newtheorem{thm}{Theorem}[section]
\newtheorem{cor}[thm]{Corollary}
\newtheorem{lem}[thm]{Lemma}

\newtheorem{defn}[thm]{Definition}

\numberwithin{equation}{section}

\begin{document}

\title{\bf $E_8$ bundles and some new anomaly cancellation formulas }
\author{Yong Wang}

\thanks{{\scriptsize
\hskip -0.4 true cm \textit{2010 Mathematics Subject Classification:}
58C20; 57R20; 53C80.
 }}

\maketitle

\begin{abstract}
 Using $E_8$ bundles, we construct some new modular forms over $SL(2,{\bf Z})$, $\Gamma^0(2)$ and $\Gamma_0(2)$ and get
 some new anomaly cancellation formulas of characteristic forms which generalize some anomaly cancellation formulas in \cite{WY}.
\end{abstract}

\vskip 0.2 true cm


\pagestyle{myheadings}
\markboth{\rightline {\scriptsize Yong Wang}}
         {\leftline{\scriptsize $E_8$ bundles and some new anomaly cancellation formulas}}

\bigskip
\bigskip


\section{ Introduction}
\quad In 1983, the physicists Alvarez-Gaum\'{e} and Witten \cite{AW}
  discovered the "miraculous cancellation" formula for gravitational
  anomaly which reveals a beautiful relation between the top
  components of the Hirzebruch $\widehat{L}$-form and
  $\widehat{A}$-form of a $12$-dimensional smooth Riemannian
  manifold. Kefeng Liu \cite{Li1} established higher dimensional "miraculous cancellation"
  formulas for $(8k+4)$-dimensional Riemannian manifolds by
  developing modular invariance properties of characteristic forms.
  These formulas could be used to deduce some divisibility results. In
  \cite{HZ1}, \cite{HZ2}, \cite{CH}, some more general cancellation formulas that involve a
  complex line bundle and their applications were established. In \cite{HLZ1}, using the Eisenstein series, a more general cancellation
  formula was derived.  In \cite{HLZ2}, Han, Liu and Zhang showed that both of the Green-Schwarz anomaly factorization formula
for the gauge group $E_8\times E_8$ and the Horava-Witten anomaly factorization formula for the gauge
group $E_8$ could be derived through modular forms of weight $14$. This answered a question of J.
H. Schwarz. They also established generalizations of these factorization formulas and obtaind a new
Horava-Witten type factorization formula on $12$-dimensional manifolds. In \cite{HHLZ}, Han, Huang, Liu and Zhang introduced a modular form of weight $14$ over $SL(2,{\bf Z})$ and a modular form of weight $10$ over $SL(2,{\bf Z})$ by $E_8$ bundles and they got some interesting
anomaly cancellation formulas on $12$-dimensional manifolds.
  In \cite{CHZ}, Chen, Han and Zhang defined an integral modular form of weight $2k$ for a $4k+2$-dimensional $spin^c$ manifold.
  In \cite{WY}, we twist the Chen-Han-Zhang $SL(2,{\bf Z})$ modular form by $E_8$ bundles and get $SL(2,{\bf Z})$ modular forms of weight $14$ and $10$ for $14$ and $10$-dimensional $spin^c$ manifolds. In \cite{Li1}, Liu introduced some $\Gamma^0(2)$ and $\Gamma_0(2)$ modular forms and got
 some interesting anomaly cancellation formulas.  We also twist the Liu's modular form by $E_8$ bundles and get $\Gamma^0(2)$ and $\Gamma_0(2)$ modular forms of weight $14$ and $10$ for a $12$-dimensional spin manifold. By these modular forms, we get
 some new anomaly cancellation formulas of characteristic forms.
 In odd dimensions, in \cite{LW}, we explored the combination of modular forms with $E_8$ bundles, with the goal of deriving corresponding new anomaly cancellation formulas.
Specifically, starting from the $SL(2,\mathbf{Z})$ modular forms established in \cite{Li1}, \cite{CHZ}, and \cite{GWL}, we twisted and generalized them using $E_{8}$ and $E_{8}\times E_{8}$ bundles in odd dimensions.
This approach enables the systematic construction of new modular forms on odd-dimensional spin and spin$^c$ manifolds, leading to new anomaly cancellation formulas. In this paper, using the $E_8$ bundles and an extra spin real vector bundle, we construct some new modular forms over $SL(2,{\bf Z})$, $\Gamma^0(2)$ and $\Gamma_0(2)$. By these modular forms, we get
 some new anomaly cancellation formulas of characteristic forms on $10$ and $14$ dimensional spin$^c$ manifolds which generalize some anomaly cancellation formulas in \cite{WY}.
 \\
  \indent This paper is organized as follows: In Section 2, we twist the $SL(2,{\bf Z})$ modular forms in \cite{WY} by $E_8$ bundles and an extra spin real vector bundle and get $SL(2,{\bf Z})$ modular forms of weight $14$ and $10$ for $14$ and $10$-dimensional $spin^c$ manifolds and we get
 some new anomaly cancellation formulas of characteristic forms. In Section 3, we twist the Liu's modular form by $E_8$ bundles and an extra spin real vector bundle and get $\Gamma^0(2)$ and $\Gamma_0(2)$ modular forms of weight $14$, $12$,$10$ and $8$ for a $10$ and $14$-dimensional spin manifold. By these modular forms, we get some new anomaly cancellation formulas of characteristic forms.

  \vskip 1 true cm

\section{$SL(2,Z)$ modular forms and anomaly cancellation formulas for $14$ and $10$-dimensional $spin^c$ manifolds}

\indent For the basic representation theory for the affine $E_8$, we can see Section 1 in \cite{HLZ2}. Here we omit it. For the principle $E_8$ bundle $P_i$, $i=1,2$, consider the associated bundles $\mathcal{V}_i=\sum_{k=0}^{\infty}(P_i\times_{\rho_k}V_k)q^k\in K(M)[[q]].$ We let $W_i=P_i\times_{\rho_1}V_1$. Let $M$ be a $14$-dimensional spinc manifold and $L$ be the complex line bundle associated to the given spinc structure on $M$. We also
consider $L$ as a real vector bundle denoted by $L_R$. \\
\indent Denote by $c=c_1(L)=2\pi\sqrt{-1}u$ the first Chern class of $L$. Let
$\varphi(\tau)=\prod_{n=1}^{\infty}(1-q^n)$, where $q=e^{2\pi\sqrt{-1}\tau}$ and $\tau\in H$ the upper half plane.
We recall the four Jacobi theta functions are
   defined as follows( cf. \cite{Ch}):
 \begin{equation}  \theta(v,\tau)=2q^{\frac{1}{8}}{\rm sin}(\pi
   v)\prod_{j=1}^{\infty}[(1-q^j)(1-e^{2\pi\sqrt{-1}v}q^j)(1-e^{-2\pi\sqrt{-1}v}q^j)],
   \end{equation}
\begin{equation}\theta_1(v,\tau)=2q^{\frac{1}{8}}{\rm cos}(\pi
   v)\prod_{j=1}^{\infty}[(1-q^j)(1+e^{2\pi\sqrt{-1}v}q^j)(1+e^{-2\pi\sqrt{-1}v}q^j)],\end{equation}
\begin{equation}\theta_2(v,\tau)=\prod_{j=1}^{\infty}[(1-q^j)(1-e^{2\pi\sqrt{-1}v}q^{j-\frac{1}{2}})
(1-e^{-2\pi\sqrt{-1}v}q^{j-\frac{1}{2}})],\end{equation}
\begin{equation}\theta_3(v,\tau)=\prod_{j=1}^{\infty}[(1-q^j)(1+e^{2\pi\sqrt{-1}v}q^{j-\frac{1}{2}})
(1+e^{-2\pi\sqrt{-1}v}q^{j-\frac{1}{2}})],\end{equation}
By (2.3)-(2.7) in \cite{HLZ2}, we have that there are formal two forms $y_l^i,1\leq l\leq 8,i=1,2$ such that
\begin{equation}
\varphi(\tau)^8{\rm ch}(\mathcal{V}_i)=\frac{1}{2}\left(\prod_{l=1}^8\theta_1(y_l^i,\tau)+\prod_{l=1}^8\theta_2(y_l^i,\tau)+\prod_{l=1}^8\theta_3(y_l^i,\tau)\right),
\end{equation}
and
\begin{equation}
\sum_{l=1}^8(2\pi\sqrt{-1}y_l^i)^2=-\frac{1}{30}c_2(W_i),
\end{equation}
where $c_2(W_i)$ denotes the second Chern class of $W_i$. Let $E_2(\tau)$ is the Eisenstein series satisfying:
\begin{equation}
E_2(\tau+1)=E_2(\tau),~~E_2(-\frac{1}{\tau})=\tau^2E_2(\tau)-\frac{6\sqrt{-1}\tau}{\pi},
\end{equation}
and
\begin{equation}
E_2(\tau)=1-24q-72q^2+O(q^3).
\end{equation}
Let $V$ be a $2\overline{l}$-dimensional real spin vector bundle on $M$ and $\triangle(V)$ be associated spinors bundle.
Let $Q_j(V),j=1,2,3$ be the virtual bundles defined as follows:
\begin{equation}
Q_1(V)=\triangle(V)\otimes
   \bigotimes _{n=1}^{\infty}\wedge_{q^n}(\widetilde{V_C});
\end{equation}
\begin{equation}
Q_2(V)=\bigotimes _{n=1}^{\infty}\wedge_{-q^{n-\frac{1}{2}}}(\widetilde{V_C});
\end{equation}
\begin{equation}
Q_3(V)=\bigotimes _{n=1}^{\infty}\wedge_{q^{n-\frac{1}{2}}}(\widetilde{V_C}).
\end{equation}
Let
$A:=p_1(M)-p_1(L_R)-3p_1(V)+\frac{1}{30}(c_2(W_i)+c_2(W_j)),$ where $p_1(M),p_1(L_R),p_1(V)$ denote the first Pontryajin classes of $M$,$L_R$ and $V$ (for definitions, see \cite{Zh}),
and \begin{align}
Q(M,P_i,P_j,V,\tau)=&\left\{e^{\frac{1}{24}E_2(\tau)A}\widehat{A}(TX){\rm exp}(\frac{c}{2}){\rm ch}\left[\bigotimes _{n=1}^{\infty}S_{q^n}(\widetilde{T_CM})
\otimes
\bigotimes _{m=1}^{\infty}\wedge_{-q^m}(\widetilde{L_C})\right]\right.\\\notag
&{\rm ch}(Q_1(V)\otimes Q_2(V)\otimes Q_3(V))
\left.\cdot\varphi(\tau)^{16}{\rm ch}(\mathcal{V}_i){\rm ch}(\mathcal{V}_j)\right\}^{(14)}.
\end{align}
Then
\begin{align}Q(M,P_i,P_j,V,\tau)=&\left\{e^{\frac{1}{24}E_2(\tau)A}\left(\prod_{j=1}^{7}\frac{x_j\theta'(0,\tau)}{\theta(x_j,\tau)}\right)
\frac{\sqrt{-1}\theta(u,\tau)}{\theta_1(0,\tau)\theta_2(0,\tau)
\theta_3(0,\tau)}\right.\\\notag
&\left.\cdot\prod_{\alpha=1}^{\overline{l}}\frac{2\theta_1(\overline{y_{\alpha}},\tau)\theta_2(\overline{y_{\alpha}},\tau)
\theta_3(\overline{y_{\alpha}},\tau)}{\theta_1(0,\tau)\theta_2(0,\tau)
\theta_3(0,\tau)}\right.\\\notag
&\cdot\frac{1}{4}\left(\prod_{l=1}^8\theta_1(y_l^i,\tau)+\prod_{l=1}^8\theta_2(y_l^i,\tau)+\prod_{l=1}^8\theta_3(y_l^i,\tau)\right)\\\notag
&\left.\cdot\left(\prod_{l=1}^8\theta_1(y_l^j,\tau)+\prod_{l=1}^8\theta_2(y_l^j,\tau)+\prod_{l=1}^8\theta_3(y_l^j,\tau)\right)
\right\}^{(14)},
\end{align}
where $\pm 2\pi\sqrt{-1}x_j,1\leq j\leq 7$ and $\pm 2\pi\sqrt{-1}\overline{y_{\alpha}},~1\leq \alpha\leq \overline{l}$ denote the Chern roots of $T_CM$ and $V_C$.
 One
has the following transformation laws of theta functions (cf. \cite{Ch} ):
\begin{equation}\theta(v,\tau+1)=e^{\frac{\pi\sqrt{-1}}{4}}\theta(v,\tau),~~\theta(v,-\frac{1}{\tau})
=\frac{1}{\sqrt{-1}}\left(\frac{\tau}{\sqrt{-1}}\right)^{\frac{1}{2}}e^{\pi\sqrt{-1}\tau
v^2}\theta(\tau v,\tau);\end{equation}
\begin{equation}\theta_1(v,\tau+1)=e^{\frac{\pi\sqrt{-1}}{4}}\theta_1(v,\tau),~~\theta_1(v,-\frac{1}{\tau})
=\left(\frac{\tau}{\sqrt{-1}}\right)^{\frac{1}{2}}e^{\pi\sqrt{-1}\tau
v^2}\theta_2(\tau v,\tau);\end{equation}
\begin{equation}\theta_2(v,\tau+1)=\theta_3(v,\tau),~~\theta_2(v,-\frac{1}{\tau})
=\left(\frac{\tau}{\sqrt{-1}}\right)^{\frac{1}{2}}e^{\pi\sqrt{-1}\tau
v^2}\theta_1(\tau v,\tau);\end{equation}
\begin{equation}\theta_3(v,\tau+1)=\theta_2(v,\tau),~~\theta_3(v,-\frac{1}{\tau})
=\left(\frac{\tau}{\sqrt{-1}}\right)^{\frac{1}{2}}e^{\pi\sqrt{-1}\tau
v^2}\theta_3(\tau v,\tau),\end{equation}
 \begin{equation}\theta'(v,\tau+1)=e^{\frac{\pi\sqrt{-1}}{4}}\theta'(v,\tau),~~
 \theta'(0,-\frac{1}{\tau})=\frac{1}{\sqrt{-1}}\left(\frac{\tau}{\sqrt{-1}}\right)^{\frac{1}{2}}
\tau\theta'(0,\tau).\end{equation}
\begin{defn} A modular form over $\Gamma$, a
 subgroup of $SL_2({\bf Z})$, is a holomorphic function $f(\tau)$ on
 $\textbf{H}$ such that
\begin{equation} f(g\tau):=f\left(\frac{a\tau+b}{c\tau+d}\right)=\chi(g)(c\tau+d)^kf(\tau),
 ~~\forall g=\left(\begin{array}{cc}
\ a & b  \\
 c & d
\end{array}\right)\in\Gamma,\end{equation}
\noindent where $\chi:\Gamma\rightarrow {\bf C}^{\star}$ is a
character of $\Gamma$. $k$ is called the weight of $f$.
\end{defn}
By (2.7) and (2.13)-(2.18), we have
\begin{lem}
$Q(M,P_i,P_j,V,\tau)$ is a modular form over $SL_2({\bf Z})$ with the weight $14$.
\end{lem}
By
\begin{equation}
\varphi(\tau)^{16}{\rm ch}(\mathcal{V}_i){\rm ch}(\mathcal{V}_j)=1+{\rm ch}(-16+W_i+W_j)q+\cdots,
\end{equation}
\begin{equation}
{\rm ch}(Q_1(V)\otimes Q_2(V)\otimes Q_3(V))={\rm ch}[\triangle(V)+q(\triangle(V)\otimes(\widetilde{V_C}+2\wedge^2\widetilde{V_C}
-\widetilde{V_C}\otimes \widetilde{V_C}))]+\cdots.
\end{equation}
We have by expanding the $q$-series:
\begin{align}
&e^{\frac{1}{24}E_2(\tau)A}\widehat{A}(TX){\rm exp}(\frac{c}{2}){\rm ch}\left[\bigotimes _{n=1}^{\infty}S_{q^n}(\widetilde{T_CM})
\otimes
\bigotimes _{m=1}^{\infty}\wedge_{-q^m}(\widetilde{L_C})\right]\\\notag
&\cdot{\rm ch}(Q_1(V)\otimes Q_2(V)\otimes Q_3(V))\varphi(\tau)^{16}{\rm ch}(\mathcal{V}_i){\rm ch}(\mathcal{V}_j)\\\notag
&=(e^{\frac{1}{24}A}-e^{\frac{1}{24}A}Aq+O(q^2))\widehat{A}(TX){\rm exp}(\frac{c}{2})
{\rm ch}\left[(1+q\widetilde{T_CM}+O(q^2))(1-q\widetilde{L_C}+O(q^2))\right]\\\notag
&({\rm ch}[\triangle(V)+q(\triangle(V)\otimes(\widetilde{V_C}+2\wedge^2\widetilde{V_C}
-\widetilde{V_C}\otimes \widetilde{V_C}))]+O(q^2))\\\notag
&\cdot(1-16q+O(q^2))(1+{\rm ch}({W_i})q+O(q^2))(1+{\rm ch}({W_j})q+O(q^2))\\\notag
&=e^{\frac{1}{24}A}\widehat{A}(TX){\rm exp}(\frac{c}{2}){\rm ch}(\triangle(V))+\left[e^{\frac{1}{24}A}\widehat{A}(TX){\rm exp}(\frac{c}{2})
{\rm ch}(\triangle(V))\right.\\\notag
&\cdot{\rm ch}(\widetilde{T_CM}-\widetilde{L_R}-16+W_i+W_j+\widetilde{V_C}+2\wedge^2\widetilde{V_C}
-\widetilde{V_C}\otimes \widetilde{V_C})\\\notag
&\left.-e^{\frac{1}{24}A}A\widehat{A}(TX){\rm exp}(\frac{c}{2}){\rm ch}(\triangle(V))\right]q+
O(q^2).\notag
\end{align}
It is well known that modular forms over $SL_2({\bf Z})$ can be expressed as polynomials of the Einsentein series $E_4(\tau)$ and $E_6(\tau)$,
where
 \begin{equation}
E_4(\tau)=1+240q+2160q^2+6720q^3+\cdots,
\end{equation}
\begin{equation}
E_6(\tau)=1-504q-16632q^2-122976q^3+\cdots.
\end{equation}
Their weights are $4$ and $6$ respectively. Since $Q(M,P_i,P_j,\tau)$ is a modular form over $SL_2({\bf Z})$ with the weight $14$, it must be
a multiple of
\begin{equation}
E_4(\tau)^2E_6(\tau)=1-24q-196632q^2+\cdots.
\end{equation}
So
\begin{align}
&\left[e^{\frac{1}{24}A}\widehat{A}(TX){\rm exp}(\frac{c}{2})
{\rm ch}(\triangle(V))\right.\\\notag
&\cdot{\rm ch}(\widetilde{T_CM}-\widetilde{L_R}-16+W_i+W_j+\widetilde{V_C}+2\wedge^2\widetilde{V_C}
-\widetilde{V_C}\otimes \widetilde{V_C})\\\notag
&\left.-e^{\frac{1}{24}A}A\widehat{A}(TX){\rm exp}(\frac{c}{2}){\rm ch}(\triangle(V))\right]^{(14)}\\\notag
&=-24\left\{e^{\frac{1}{24}A}\widehat{A}(TX){\rm exp}(\frac{c}{2}){\rm ch}(\triangle(V))\right\}^{(14)}. \notag
\end{align}
By (2.26), we have
\begin{thm} One has the following equality:
\begin{align}
&\left\{\widehat{A}(TX){\rm exp}(\frac{c}{2})
{\rm ch}(\triangle(V))\right.\\\notag
&\left.\cdot{\rm ch}(\widetilde{T_CM}-\widetilde{L_R}+8+W_i+W_j+\widetilde{V_C}+2\wedge^2\widetilde{V_C}
-\widetilde{V_C}\otimes \widetilde{V_C})\right\}^{(14)}\\\notag
&=A\left\{e^{\frac{1}{24}A}\widehat{A}(TX){\rm exp}(\frac{c}{2}){\rm ch}(\triangle(V))-\frac{e^{\frac{1}{24}A}-1}{A}
\widehat{A}(TX){\rm exp}(\frac{c}{2})
{\rm ch}(\triangle(V))\right.\\\notag
&\left.\cdot{\rm ch}(\widetilde{T_CM}-\widetilde{L_R}+8+W_i+W_j+\widetilde{V_C}+2\wedge^2\widetilde{V_C}
-\widetilde{V_C}\otimes \widetilde{V_C})\right\}^{(10)}.\notag
\end{align}
\end{thm}
\begin{cor}
When $A=0$, we have
\begin{align}
&\left[\widehat{A}(TX){\rm exp}(\frac{c}{2})
{\rm ch}(\triangle(V))\right.\\\notag
&\left.\cdot{\rm ch}(\widetilde{T_CM}-\widetilde{L_R}-16+W_i+W_j+\widetilde{V_C}+2\wedge^2\widetilde{V_C}
-\widetilde{V_C}\otimes \widetilde{V_C})
\right]^{(14)}\\\notag
&=-24\left\{\widehat{A}(TX){\rm exp}(\frac{c}{2}){\rm ch}(\triangle(V))\right\}^{(14)}. \notag
\end{align}
When $M$ is a $14$-dimensional $spin^c$ manifold, then ${\rm Ind}D^c\otimes(\triangle(V))\otimes(\widetilde{T_CM}-\widetilde{L_R}-16+W_i+W_j+\widetilde{V_C}+2\wedge^2\widetilde{V_C}
-\widetilde{V_C}\otimes \widetilde{V_C})_+$ is a multiply of $24$.
\end{cor}
Let
$A_1:=p_1(M)-p_1(L_R)-3p_1(V)+\frac{1}{30}c_2(W_i).$
Let \begin{align}
Q(M,P_i,V,\tau)=&\left\{e^{\frac{1}{24}E_2(\tau)A_1}\widehat{A}(TX){\rm exp}(\frac{c}{2}){\rm ch}\left[\bigotimes _{n=1}^{\infty}S_{q^n}(\widetilde{T_CM})
\otimes
\bigotimes _{m=1}^{\infty}\wedge_{-q^m}(\widetilde{L_C})\right]\right.\\\notag
&{\rm ch}(Q_1(V)\otimes Q_2(V)\otimes Q_3(V))
\left.\cdot\varphi(\tau)^{8}{\rm ch}(\mathcal{V}_i)\right\}^{(14)}.
\end{align}
Then
\begin{align}Q(M,P_i,V,\tau)=&\left\{e^{\frac{1}{24}E_2(\tau)A_1}\left(\prod_{j=1}^{7}\frac{x_j\theta'(0,\tau)}{\theta(x_j,\tau)}\right)
\frac{\sqrt{-1}\theta(u,\tau)}{\theta_1(0,\tau)\theta_2(0,\tau)
\theta_3(0,\tau)}\right.\\\notag
&\left.\cdot\prod_{\alpha=1}^{\overline{l}}\frac{2\theta_1(\overline{y_{\alpha}},\tau)\theta_2(\overline{y_{\alpha}},\tau)
\theta_3(\overline{y_{\alpha}},\tau)}{\theta_1(0,\tau)\theta_2(0,\tau)
\theta_3(0,\tau)}\right.\\\notag
&\cdot\frac{1}{2}\left(\prod_{l=1}^8\theta_1(y_l^i,\tau)+\prod_{l=1}^8\theta_2(y_l^i,\tau)+\prod_{l=1}^8\theta_3(y_l^i,\tau)
\right\}^{(14)},
\end{align}
Similar to Lemma 2.2, we have
\begin{lem}
$Q(M,P_i,V,\tau)$ is a modular form over $SL_2({\bf Z})$ with the weight $10$.
\end{lem}
Similar to (2.22), we have
\begin{align}
&e^{\frac{1}{24}E_2(\tau)A_1}\widehat{A}(TX){\rm exp}(\frac{c}{2}){\rm ch}\left[\bigotimes _{n=1}^{\infty}S_{q^n}(\widetilde{T_CM})
\otimes
\bigotimes _{m=1}^{\infty}\wedge_{-q^m}(\widetilde{L_C})\right]\\\notag
&\cdot{\rm ch}(Q_1(V)\otimes Q_2(V)\otimes Q_3(V))\varphi(\tau)^{8}{\rm ch}(\mathcal{V}_i)\\\notag
&=(e^{\frac{1}{24}A_1}-e^{\frac{1}{24}A_1}A_1q+O(q^2))\widehat{A}(TX){\rm exp}(\frac{c}{2})
{\rm ch}\left[(1+q\widetilde{T_CM}+O(q^2))(1-q\widetilde{L_C}+O(q^2))\right]\\\notag
&({\rm ch}[\triangle(V)+q(\triangle(V)\otimes(\widetilde{V_C}+2\wedge^2\widetilde{V_C}
-\widetilde{V_C}\otimes \widetilde{V_C}))]+O(q^2))\\\notag
&\cdot(1-8q+O(q^2))(1+{\rm ch}({W_i})q+O(q^2))\\\notag
&=e^{\frac{1}{24}A_1}\widehat{A}(TX){\rm exp}(\frac{c}{2}){\rm ch}(\triangle(V))+\left[e^{\frac{1}{24}A_1}\widehat{A}(TX){\rm exp}(\frac{c}{2})
{\rm ch}(\triangle(V))\right.\\\notag
&\cdot{\rm ch}(\widetilde{T_CM}-\widetilde{L_R}-8+W_i+\widetilde{V_C}+2\wedge^2\widetilde{V_C}
-\widetilde{V_C}\otimes \widetilde{V_C})\\\notag
&\left.-e^{\frac{1}{24}A_1}A_1\widehat{A}(TX){\rm exp}(\frac{c}{2}){\rm ch}(\triangle(V))\right]q+
O(q^2).\notag
\end{align}
Since $Q(M,P_i,V,\tau)$ is a modular form over $SL_2({\bf Z})$ with the weight $10$, it must be
a multiple of
\begin{equation}
E_4(\tau)E_6(\tau)=1-264q-135432q^2+\cdots.
\end{equation}
So
\begin{align}
&\left\{e^{\frac{1}{24}A_1}\widehat{A}(TX){\rm exp}(\frac{c}{2})
{\rm ch}(\triangle(V))\right.\\\notag
&\cdot{\rm ch}(\widetilde{T_CM}-\widetilde{L_R}-8+W_i+\widetilde{V_C}+2\wedge^2\widetilde{V_C}
-\widetilde{V_C}\otimes \widetilde{V_C})\\\notag
&\left.-e^{\frac{1}{24}A_1}A_1\widehat{A}(TX){\rm exp}(\frac{c}{2}){\rm ch}(\triangle(V))\right\}^{(14)}\\\notag
&=-264\left\{e^{\frac{1}{24}A_1}\widehat{A}(TX){\rm exp}(\frac{c}{2}){\rm ch}(\triangle(V))\right\}^{(14)}.
\end{align}
\begin{thm} One has the following equality:
\begin{align}
&\left\{\widehat{A}(TX){\rm exp}(\frac{c}{2})
{\rm ch}(\triangle(V))\right.\\\notag
&\left.\cdot{\rm ch}(\widetilde{T_CM}-\widetilde{L_R}+256+W_i+\widetilde{V_C}+2\wedge^2\widetilde{V_C}
-\widetilde{V_C}\otimes \widetilde{V_C})\right\}^{(14)}\\\notag
&=A_1\left\{e^{\frac{1}{24}A_1}\widehat{A}(TX){\rm exp}(\frac{c}{2}){\rm ch}(\triangle(V))-\frac{e^{\frac{1}{24}A_1}-1}{A_1}
\widehat{A}(TX){\rm exp}(\frac{c}{2})
{\rm ch}(\triangle(V))\right.\\\notag
&\left.\cdot{\rm ch}(\widetilde{T_CM}-\widetilde{L_R}+256+W_i+\widetilde{V_C}+2\wedge^2\widetilde{V_C}
-\widetilde{V_C}\otimes \widetilde{V_C})\right\}^{(10)}.\notag
\end{align}
\end{thm}
\begin{cor}
When $A_1=0$, we have
\begin{align}
&\left\{\widehat{A}(TX){\rm exp}(\frac{c}{2})
{\rm ch}(\triangle(V))\right.\\\notag
&\left.\cdot{\rm ch}(\widetilde{T_CM}-\widetilde{L_R}-8+W_i+\widetilde{V_C}+2\wedge^2\widetilde{V_C}
-\widetilde{V_C}\otimes \widetilde{V_C})\right\}^{(14)}\\\notag
&=-264\left\{\widehat{A}(TX){\rm exp}(\frac{c}{2}){\rm ch}(\triangle(V))\right\}^{(14)}
. \notag
\end{align}
When $M$ is a $14$-dimensional $spin^c$ manifold, then ${\rm Ind}D^c\otimes(\triangle(V))\otimes
(\widetilde{T_CM}-\widetilde{L_R}-8+W_i+\widetilde{V_C}+2\wedge^2\widetilde{V_C}
-\widetilde{V_C}\otimes \widetilde{V_C})_+$ is a multiply of $264$.
\end{cor}
If $M$ is a $10$-dimensional $spin^c$ manifold, we let
Let \begin{align}
R(M,P_i,V,\tau)=&\left\{e^{\frac{1}{24}E_2(\tau)A_1}\widehat{A}(TX){\rm exp}(\frac{c}{2}){\rm ch}\left[\bigotimes _{n=1}^{\infty}S_{q^n}(\widetilde{T_CM})
\otimes
\bigotimes _{m=1}^{\infty}\wedge_{-q^m}(\widetilde{L_C})\right]\right.\\\notag
&{\rm ch}(Q_1(V)\otimes Q_2(V)\otimes Q_3(V))
\left.\cdot\varphi(\tau)^{8}{\rm ch}(\mathcal{V}_i)\right\}^{(10)}.
\end{align}
Then
\begin{align}R(M,P_i,V,\tau)=&\left\{e^{\frac{1}{24}E_2(\tau)A_1}\left(\prod_{j=1}^{5}\frac{x_j\theta'(0,\tau)}{\theta(x_j,\tau)}\right)
\frac{\sqrt{-1}\theta(u,\tau)}{\theta_1(0,\tau)\theta_2(0,\tau)
\theta_3(0,\tau)}\right.\\\notag
&\left.\cdot\prod_{\alpha=1}^{\overline{l}}\frac{2\theta_1(\overline{y_{\alpha}},\tau)\theta_2(\overline{y_{\alpha}},\tau)
\theta_3(\overline{y_{\alpha}},\tau)}{\theta_1(0,\tau)\theta_2(0,\tau)
\theta_3(0,\tau)}\right.\\\notag
&\cdot\frac{1}{2}\left(\prod_{l=1}^8\theta_1(y_l^i,\tau)+\prod_{l=1}^8\theta_2(y_l^i,\tau)+\prod_{l=1}^8\theta_3(y_l^i,\tau)
\right\}^{(10)},
\end{align}
Similar to Lemma 2.2, we have
\begin{lem}
$R(M,P_i,V,\tau)$ is a modular form over $SL_2({\bf Z})$ with the weight $8$.
\end{lem}
It must be
a multiple of
\begin{equation}
E_4(\tau)^2=1+480q+61920q^2+\cdots.
\end{equation}
So
\begin{align}
&\left\{e^{\frac{1}{24}A_1}\widehat{A}(TX){\rm exp}(\frac{c}{2})
{\rm ch}(\triangle(V))\right.\\\notag
&\cdot{\rm ch}(\widetilde{T_CM}-\widetilde{L_R}-8+W_i+\widetilde{V_C}+2\wedge^2\widetilde{V_C}
-\widetilde{V_C}\otimes \widetilde{V_C})\\\notag
&\left.-e^{\frac{1}{24}A_1}A_1\widehat{A}(TX){\rm exp}(\frac{c}{2}){\rm ch}(\triangle(V))\right\}^{(10)}\\\notag
&=480\left\{e^{\frac{1}{24}A_1}\widehat{A}(TX){\rm exp}(\frac{c}{2}){\rm ch}(\triangle(V))\right\}^{(10)}.
\end{align}
\begin{thm} One has the following equality:
\begin{align}
&\left\{\widehat{A}(TX){\rm exp}(\frac{c}{2})
{\rm ch}(\triangle(V))\right.\\\notag
&\left.\cdot{\rm ch}(\widetilde{T_CM}-\widetilde{L_R}-488+W_i+\widetilde{V_C}+2\wedge^2\widetilde{V_C}
-\widetilde{V_C}\otimes \widetilde{V_C})\right\}^{(10)}\\\notag
&=A_1\left\{e^{\frac{1}{24}A_1}\widehat{A}(TX){\rm exp}(\frac{c}{2}){\rm ch}(\triangle(V))-\frac{e^{\frac{1}{24}A_1}-1}{A_1}
\widehat{A}(TX){\rm exp}(\frac{c}{2})
{\rm ch}(\triangle(V))\right.\\\notag
&\left.\cdot{\rm ch}(\widetilde{T_CM}-\widetilde{L_R}-488+W_i+\widetilde{V_C}+2\wedge^2\widetilde{V_C}
-\widetilde{V_C}\otimes \widetilde{V_C})\right\}^{(6)}.\notag
\end{align}
\end{thm}
\begin{cor}
When $A_1=0$, we have
\begin{align}
&\left\{\widehat{A}(TX){\rm exp}(\frac{c}{2})
{\rm ch}(\triangle(V))\right.\\\notag
&\left.\cdot{\rm ch}(\widetilde{T_CM}-\widetilde{L_R}+W_i+\widetilde{V_C}+2\wedge^2\widetilde{V_C}
-\widetilde{V_C}\otimes \widetilde{V_C})\right\}^{(10)}\\\notag
&=488\left\{\widehat{A}(TX){\rm exp}(\frac{c}{2}){\rm ch}(\triangle(V))\right\}^{(10)}
. \notag
\end{align}
When $M$ is a $10$-dimensional $spin^c$ manifold, then ${\rm Ind}D^c\otimes(\triangle(V))\otimes
(\widetilde{T_CM}-\widetilde{L_R}+W_i+\widetilde{V_C}+2\wedge^2\widetilde{V_C}
-\widetilde{V_C}\otimes \widetilde{V_C})_+$ is a multiply of $488$.
\end{cor}

When $M$ is a $14$-dimensional $spin^c$ manifold and $A_2:=p_1(M)-p_1(L_R)-p_1(V)+\frac{1}{30}(c_2(W_i)+c_2(W_j)),$
let \begin{align}
\widetilde{Q}(M,P_i,P_j,V,\tau)=&\left\{e^{\frac{1}{24}E_2(\tau)A_2}\widehat{A}(TX){\rm exp}(\frac{c}{2}){\rm ch}\left[\bigotimes _{n=1}^{\infty}S_{q^n}(\widetilde{T_CM})
\otimes
\bigotimes _{m=1}^{\infty}\wedge_{-q^m}(\widetilde{L_C})\right]\right.\\\notag
&{\rm ch}(Q_1(V)+2^{\overline{l}}Q_2(V)+2^{\overline{l}} Q_3(V))
\left.\cdot\varphi(\tau)^{16}{\rm ch}(\mathcal{V}_i){\rm ch}(\mathcal{V}_j)\right\}^{(14)}.
\end{align}
Then
\begin{align}\widetilde{Q}(M,P_i,P_j,V,\tau)=&\left\{e^{\frac{1}{24}E_2(\tau)A_2}\left(\prod_{j=1}^{7}\frac{x_j\theta'(0,\tau)}{\theta(x_j,\tau)}\right)
\frac{\sqrt{-1}\theta(u,\tau)}{\theta_1(0,\tau)\theta_2(0,\tau)
\theta_3(0,\tau)}\right.\\\notag
&\left.\cdot 2^{\overline{l}}\left(\prod_{\alpha=1}^{\overline{l}}\frac{\theta_1(\overline{y_{\alpha}},\tau)}{\theta_1(0,\tau)}
+\prod_{\alpha=1}^{\overline{l}}\frac{\theta_2(\overline{y_{\alpha}},\tau)}{\theta_2(0,\tau)}
+\prod_{\alpha=1}^{\overline{l}}\frac{\theta_3(\overline{y_{\alpha}},\tau)}{\theta_3(0,\tau)}\right)
\right.\\\notag
&\cdot\frac{1}{4}\left(\prod_{l=1}^8\theta_1(y_l^i,\tau)+\prod_{l=1}^8\theta_2(y_l^i,\tau)+\prod_{l=1}^8\theta_3(y_l^i,\tau)\right)\\\notag
&\left.\cdot\left(\prod_{l=1}^8\theta_1(y_l^j,\tau)+\prod_{l=1}^8\theta_2(y_l^j,\tau)+\prod_{l=1}^8\theta_3(y_l^j,\tau)\right)
\right\}^{(14)},
\end{align}
Then
\begin{lem}
$\widetilde{Q}(M,P_i,P_j,V,\tau)$ is a modular form over $SL_2({\bf Z})$ with the weight $14$.
\end{lem}
We have
\begin{align}
\widetilde{Q}(M,P_i,P_j,V,\tau)
&=[e^{\frac{1}{24}A_2}\widehat{A}(TX){\rm exp}(\frac{c}{2}){\rm ch}(\triangle(V)+2)]^{(14)}+\left[e^{\frac{1}{24}A_2}\widehat{A}(TX){\rm exp}(\frac{c}{2})
\right.\\\notag
&\cdot{\rm ch}(\widetilde{T_CM}-\widetilde{L_R}-16+W_i+W_j)\otimes(\triangle(V)+2)+\triangle(V)\otimes\widetilde{V_C}+2\wedge^2\widetilde{V_C}
)\\\notag
&\left.-e^{\frac{1}{24}A_2}A_2\widehat{A}(TX){\rm exp}(\frac{c}{2}){\rm ch}(\triangle(V)+2)\right]^{(14)}q+
O(q^2).\notag
\end{align}
So
\begin{align}
&\left[e^{\frac{1}{24}A_2}\widehat{A}(TX){\rm exp}(\frac{c}{2})
\right.\\\notag
&\cdot{\rm ch}(\widetilde{T_CM}-\widetilde{L_R}-16+W_i+W_j)\otimes(\triangle(V)+2)+\triangle(V)\otimes\widetilde{V_C}+2\wedge^2\widetilde{V_C}
)\\\notag
&\left.-e^{\frac{1}{24}A_2}A_2\widehat{A}(TX){\rm exp}(\frac{c}{2}){\rm ch}(\triangle(V)+2)\right]^{(14)}\\\notag
&=-24[e^{\frac{1}{24}A_2}\widehat{A}(TX){\rm exp}(\frac{c}{2}){\rm ch}(\triangle(V)+2)]^{(14)}.
\end{align}
\begin{thm} One has the following equality:
\begin{align}
&\left\{\widehat{A}(TX){\rm exp}(\frac{c}{2})
\right.\\\notag
&\left.\cdot{\rm ch}(\widetilde{T_CM}-\widetilde{L_R}+8+W_i+W_j)\otimes(\triangle(V)+2)+\triangle(V)\otimes\widetilde{V_C}+2\wedge^2\widetilde{V_C}
)\right\}^{(14)}\\\notag
&=A_2\left\{e^{\frac{1}{24}A_2}\widehat{A}(TX){\rm exp}(\frac{c}{2}){\rm ch}(\triangle(V)+2)-\frac{e^{\frac{1}{24}A_2}-1}{A_2}
\widehat{A}(TX){\rm exp}(\frac{c}{2})\right.\\\notag
&\left.\cdot{\rm ch}(\widetilde{T_CM}-\widetilde{L_R}+8+W_i+W_j)\otimes(\triangle(V)+2)+\triangle(V)\otimes\widetilde{V_C}+2\wedge^2\widetilde{V_C}
)\right\}^{(10)}.\notag
\end{align}
\end{thm}
\begin{cor}
When $A_2=0$, we have
\begin{align}
&\left[\widehat{A}(TX){\rm exp}(\frac{c}{2})
\right.\\\notag
&\cdot\left.{\rm ch}((\widetilde{T_CM}-\widetilde{L_R}-16+W_i+W_j)\otimes(\triangle(V)+2)+\triangle(V)\otimes\widetilde{V_C}+2\wedge^2\widetilde{V_C}
)\right\}^{(14)}
\\\notag
&=-24\left\{\widehat{A}(TX){\rm exp}(\frac{c}{2}){\rm ch}(\triangle(V)+2)\right\}^{(14)}. \notag
\end{align}
When $M$ is a $14$-dimensional $spin^c$ manifold, then ${\rm Ind}D^c\otimes((\widetilde{T_CM}-\widetilde{L_R}-16+W_i+W_j)\otimes(\triangle(V)+2)+\triangle(V)\otimes\widetilde{V_C}+2\wedge^2\widetilde{V_C}
))_+$ is a multiply of $24$.
\end{cor}
Let
$A_3:=p_1(M)-p_1(L_R)-p_1(V)+\frac{1}{30}c_2(W_i)$ and
 \begin{align}
\widetilde{Q}(M,P_i,V,\tau)=&\left\{e^{\frac{1}{24}E_2(\tau)A_3}\widehat{A}(TX){\rm exp}(\frac{c}{2}){\rm ch}\left[\bigotimes _{n=1}^{\infty}S_{q^n}(\widetilde{T_CM})
\otimes
\bigotimes _{m=1}^{\infty}\wedge_{-q^m}(\widetilde{L_C})\right]\right.\\\notag
&{\rm ch}(Q_1(V)+2^{\overline{l}}Q_2(V)+2^{\overline{l}} Q_3(V))\\\notag
&\left.\cdot\varphi(\tau)^{8}{\rm ch}(\mathcal{V}_i)\right\}^{(14)}.
\end{align}
Then
\begin{align}\widetilde{Q}(M,P_i,V,\tau)=&\left\{e^{\frac{1}{24}E_2(\tau)A_1}\left(\prod_{j=1}^{7}\frac{x_j\theta'(0,\tau)}{\theta(x_j,\tau)}\right)
\frac{\sqrt{-1}\theta(u,\tau)}{\theta_1(0,\tau)\theta_2(0,\tau)
\theta_3(0,\tau)}\right.\\\notag
&\left.\cdot 2^{\overline{l}}\left(\prod_{\alpha=1}^{\overline{l}}\frac{\theta_1(\overline{y_{\alpha}},\tau)}{\theta_1(0,\tau)}
+\prod_{\alpha=1}^{\overline{l}}\frac{\theta_2(\overline{y_{\alpha}},\tau)}{\theta_2(0,\tau)}
+\prod_{\alpha=1}^{\overline{l}}\frac{\theta_3(\overline{y_{\alpha}},\tau)}{\theta_3(0,\tau)}\right)
\right.\\\notag
&\cdot\frac{1}{2}\left(\prod_{l=1}^8\theta_1(y_l^i,\tau)+\prod_{l=1}^8\theta_2(y_l^i,\tau)+\prod_{l=1}^8\theta_3(y_l^i,\tau)
\right\}^{(14)},
\end{align}
Similar to Lemma 2.2, we have
\begin{lem}
$\widetilde{Q}(M,P_i,V,\tau)$ is a modular form over $SL_2({\bf Z})$ with the weight $10$.
\end{lem}
We have
\begin{align}
\widetilde{Q}(M,P_i,V,\tau)
&=[e^{\frac{1}{24}A_3}\widehat{A}(TX){\rm exp}(\frac{c}{2}){\rm ch}(\triangle(V)+2)]^{(14)}+\left[e^{\frac{1}{24}A_3}\widehat{A}(TX){\rm exp}(\frac{c}{2})
\right.\\\notag
&\cdot{\rm ch}((\widetilde{T_CM}-\widetilde{L_R}-8+W_i)\otimes(\triangle(V)+2)+\triangle(V)\otimes\widetilde{V_C}+2\wedge^2\widetilde{V_C}
)\\\notag
&\left.-e^{\frac{1}{24}A_3}A_3\widehat{A}(TX){\rm exp}(\frac{c}{2}){\rm ch}(\triangle(V)+2)\right]^{(14)}q+
O(q^2).\notag
\end{align}
So
\begin{align}
&\left[e^{\frac{1}{24}A_3}\widehat{A}(TX){\rm exp}(\frac{c}{2})
\right.\\\notag
&\cdot{\rm ch}((\widetilde{T_CM}-\widetilde{L_R}-8+W_i)\otimes(\triangle(V)+2)+\triangle(V)\otimes\widetilde{V_C}+2\wedge^2\widetilde{V_C}
)\\\notag
&\left.-e^{\frac{1}{24}A_3}A_3\widehat{A}(TX){\rm exp}(\frac{c}{2}){\rm ch}(\triangle(V)+2)\right]^{(14)}\\\notag
&=-264[e^{\frac{1}{24}A_3}\widehat{A}(TX){\rm exp}(\frac{c}{2}){\rm ch}(\triangle(V)+2)]^{(14)}.
\end{align}

\begin{thm} One has the following equality:
\begin{align}
&\left\{\widehat{A}(TX){\rm exp}(\frac{c}{2})
\right.\\\notag
&\left.\cdot{\rm ch}((\widetilde{T_CM}-\widetilde{L_R}+252+W_i)\otimes(\triangle(V)+2)+\triangle(V)\otimes\widetilde{V_C}+2\wedge^2\widetilde{V_C}
)\right\}^{(14)}\\\notag
&=A_3\left\{e^{\frac{1}{24}A_3}\widehat{A}(TX){\rm exp}(\frac{c}{2}){\rm ch}(\triangle(V)+2)-\frac{e^{\frac{1}{24}A_3}-1}{A_3}
\widehat{A}(TX){\rm exp}(\frac{c}{2})\right.\\\notag
&\left.\cdot{\rm ch}((\widetilde{T_CM}-\widetilde{L_R}+252+W_i)\otimes(\triangle(V)+2)+\triangle(V)\otimes\widetilde{V_C}+2\wedge^2\widetilde{V_C}
)\right\}^{(10)}.\notag
\end{align}
\end{thm}
\begin{cor}
When $A_3=0$, we have
\begin{align}
&\left[\widehat{A}(TX){\rm exp}(\frac{c}{2})
\right.\\\notag
&\cdot\left.{\rm ch}((\widetilde{T_CM}-\widetilde{L_R}-8+W_i)\otimes(\triangle(V)+2)+\triangle(V)\otimes\widetilde{V_C}+2\wedge^2\widetilde{V_C}
)\right\}^{(14)}
\\\notag
&=-264\left\{\widehat{A}(TX){\rm exp}(\frac{c}{2}){\rm ch}(\triangle(V)+2)\right\}^{(14)}. \notag
\end{align}
When $M$ is a $14$-dimensional $spin^c$ manifold, then ${\rm Ind}D^c\otimes((\widetilde{T_CM}-\widetilde{L_R}-8+W_i)\otimes(\triangle(V)+2)+\triangle(V)\otimes\widetilde{V_C}+2\wedge^2\widetilde{V_C}
))_+$ is a multiply of $264$.
\end{cor}

When $M$ is a $10$-dimensional $spin^c$ manifold and let

\begin{align}
\widetilde{R}(M,P_i,V,\tau)=&\left\{e^{\frac{1}{24}E_2(\tau)A_3}\widehat{A}(TX){\rm exp}(\frac{c}{2}){\rm ch}\left[\bigotimes _{n=1}^{\infty}S_{q^n}(\widetilde{T_CM})
\otimes
\bigotimes _{m=1}^{\infty}\wedge_{-q^m}(\widetilde{L_C})\right]\right.\\\notag
&{\rm ch}(Q_1(V)+2^{\overline{l}}Q_2(V)+2^{\overline{l}} Q_3(V))\\\notag
&\left.\cdot\varphi(\tau)^{8}{\rm ch}(\mathcal{V}_i)\right\}^{(10)}.
\end{align}
Then
\begin{align}\widetilde{R}(M,P_i,V,\tau)=&\left\{e^{\frac{1}{24}
E_2(\tau)A_1}\left(\prod_{j=1}^{5}\frac{x_j\theta'(0,\tau)}{\theta(x_j,\tau)}\right)
\frac{\sqrt{-1}\theta(u,\tau)}{\theta_1(0,\tau)\theta_2(0,\tau)
\theta_3(0,\tau)}\right.\\\notag
&\left.\cdot 2^{\overline{l}}\left(\prod_{\alpha=1}^{\overline{l}}\frac{\theta_1(\overline{y_{\alpha}},\tau)}{\theta_1(0,\tau)}
+\prod_{\alpha=1}^{\overline{l}}\frac{\theta_2(\overline{y_{\alpha}},\tau)}{\theta_2(0,\tau)}
+\prod_{\alpha=1}^{\overline{l}}\frac{\theta_3(\overline{y_{\alpha}},\tau)}{\theta_3(0,\tau)}\right)
\right.\\\notag
&\cdot\frac{1}{2}\left(\prod_{l=1}^8\theta_1(y_l^i,\tau)+\prod_{l=1}^8\theta_2(y_l^i,\tau)+\prod_{l=1}^8\theta_3(y_l^i,\tau)
\right\}^{(10)},
\end{align}
\begin{lem}
$\widetilde{R}(M,P_i,V,\tau)$ is a modular form over $SL_2({\bf Z})$ with the weight $8$.
\end{lem}
\begin{cor}
When $A_3=0$, we have
\begin{align}
&\left[\widehat{A}(TX){\rm exp}(\frac{c}{2})
\right.\\\notag
&\cdot\left.{\rm ch}((\widetilde{T_CM}-\widetilde{L_R}+W_i)\otimes(\triangle(V)+2)+\triangle(V)\otimes\widetilde{V_C}+2\wedge^2\widetilde{V_C}
)\right\}^{(10)}
\\\notag
&=488\left\{\widehat{A}(TX){\rm exp}(\frac{c}{2}){\rm ch}(\triangle(V)+2)\right\}^{(10)}. \notag
\end{align}
When $M$ is a $10$-dimensional $spin^c$ manifold, then ${\rm Ind}D^c\otimes((\widetilde{T_CM}-\widetilde{L_R}+W_i)\otimes(\triangle(V)+2)+\triangle(V)\otimes\widetilde{V_C}+2\wedge^2\widetilde{V_C}
))_+$ is a multiply of $488$.
\end{cor}

\section{$\Gamma^0(2)$ modular forms and anomaly cancellation formulas for $14$ and $10$-dimensional $spin^c$ manifolds}

Let $M$ be a $10$-dimensional spin$^c$ manifold and let
\begin{align}
Q_j(M,P_i,V,\tau)=&\left\{e^{\frac{1}{24}E_2(\tau)A_3}\widehat{A}(TX){\rm exp}(\frac{c}{2}){\rm ch}\left[\bigotimes _{n=1}^{\infty}S_{q^n}(\widetilde{T_CM})
\otimes
\bigotimes _{m=1}^{\infty}\wedge_{-q^m}(\widetilde{L_C})\right]\right.\\\notag
&\left.\cdot{\rm ch}(Q_j(V))\varphi(\tau)^{8}{\rm ch}(\mathcal{V}_i)\right\}^{(10)},~~1\leq j \leq 3.
\end{align}
Then
\begin{align}Q_1(M,P_i,V,\tau)=&\left\{e^{\frac{1}{24}E_2(\tau)A_3}\left(\prod_{j=1}^{5}\frac{x_j\theta'(0,\tau)}{\theta(x_j,\tau)}\right)
\frac{\sqrt{-1}\theta(u,\tau)}{\theta_1(0,\tau)\theta_2(0,\tau)
\theta_3(0,\tau)}\right.\\\notag
&\left.2^{\overline{l}}\cdot\prod_{\alpha=1}^{\overline{l}}\frac{\theta_1(\overline{y_{\alpha}},\tau)
}{\theta_1(0,\tau)}\frac{1}{2}\left(\prod_{l=1}^8\theta_1(y_l^i,\tau)+\prod_{l=1}^8\theta_2(y_l^i,\tau)+\prod_{l=1}^8\theta_3(y_l^i,\tau)\right)
\right\}^{(10)},
\end{align}
\begin{align}Q_j(M,P_i,V,\tau)=&\left\{e^{\frac{1}{24}E_2(\tau)A_3}\left(\prod_{j=1}^{5}\frac{x_j\theta'(0,\tau)}{\theta(x_j,\tau)}\right)
\frac{\sqrt{-1}\theta(u,\tau)}{\theta_1(0,\tau)\theta_2(0,\tau)
\theta_3(0,\tau)}\right.\\\notag
&\left.\cdot\prod_{\alpha=1}^{\overline{l}}\frac{\theta_j(\overline{y_{\alpha}},\tau)
}{\theta_j(0,\tau)}\frac{1}{2}\left(\prod_{l=1}^8\theta_1(y_l^i,\tau)+\prod_{l=1}^8\theta_2(y_l^i,\tau)+\prod_{l=1}^8\theta_3(y_l^i,\tau)\right)
\right\}^{(10)},
\end{align}
where $2\leq j \leq 3$.

Let $$\Gamma_0(2)=\left\{\left(\begin{array}{cc}
\ a & b  \\
 c  & d
\end{array}\right)\in SL_2({\bf Z})\mid c\equiv 0~({\rm
mod}~2)\right\},$$
$$\Gamma^0(2)=\left\{\left(\begin{array}{cc}
\ a & b  \\
 c  & d
\end{array}\right)\in SL_2({\bf Z})\mid b\equiv 0~({\rm
mod}~2)\right\},$$  be the two modular subgroups of $SL_2({\bf Z})$.
Writing $\theta_j=\theta_j(0,\tau),~1\leq
j\leq 3,$ we introduce four explicit modular forms (\cite{Li1}),
\begin{align}
&\delta_1(\tau)=\frac{1}{8}(\theta_2^4+\theta_3^4),~~\varepsilon_1(\tau)=\frac{1}{16}\theta_2^4\theta_3^4,\\\notag
&\delta_2(\tau)=-\frac{1}{8}(\theta_1^4+\theta_3^4),~~\varepsilon_2(\tau)=\frac{1}{16}\theta_1^4\theta_3^4.\notag
\end{align}
\noindent They have the following Fourier expansions in
$q^{\frac{1}{2}}$: \begin{align}
&\delta_1(\tau)=\frac{1}{4}+6q+6q^2+\cdots,~~\varepsilon_1(\tau)=\frac{1}{16}-q+7q^2+\cdots,\\\notag
&8\delta_2(\tau)=-1-24q^{\frac{1}{2}}-24q-96q^{\frac{3}{2}}+\cdots,~~\varepsilon_2(\tau)=q^{\frac{1}{2}}+8q+28q^{\frac{3}{2}}+\cdots.\notag
\end{align}
 They also satisfy the
transformation laws,
\begin{align}\delta_2(-\frac{1}{\tau})=\tau^2\delta_1(\tau),~~~~~~\varepsilon_2(-\frac{1}{\tau})
=\tau^4\varepsilon_1(\tau),
\end{align}
\begin{lem}(\cite{Li1})$\delta_1(\tau)$ (resp.
$\varepsilon_1(\tau)$) is a modular form of weight $2$ (resp. $4$)
over $\Gamma_0(2)$, $\delta_2(\tau)$ (resp. $\varepsilon_2(\tau)$)
is a modular form of weight $2$ (resp. $4$) over $\Gamma^0(2)$ and
moreover ${\mathcal{M}}_{{\bf R}}(\Gamma^0(2))={\bf
R}[\delta_2(\tau),\varepsilon_2(\tau)]$.
\end{lem}
By (2.14)-(2.18), we have
\begin{lem}
${Q_1}(M,P_i,\tau)$ is a
modular form of weight $8$ over $\Gamma_0(2)$, while ${Q_2}(M,P_i,\tau)$ is
a modular form of weight $8$ over $\Gamma^0(2)$ . Moreover, the
following identity holds,
\begin{align}
{Q_1}(M,P_i,-\frac{1}{\tau})=2^{\overline{l}}\tau^{8}{Q_2}(M,P_i,\tau).
\end{align}
\end{lem}
\begin{thm} We have the following equality:
\begin{align}
&\left\{e^{\frac{1}{24}A_3}\widehat{A}(TX){\rm exp}(\frac{c}{2}){\rm ch}(\triangle(V))\right\}^{(10)}=
2^{\overline{l}-8}\left\{e^{\frac{1}{24}A_3}\widehat{A}(TM){\rm exp}(\frac{c}{2})\right.\\\notag
&\left.\cdot{\rm ch}[-252-8\widetilde{V_C}+\widetilde{T_CM}-\widetilde{L_R}+W_i+\wedge^2\widetilde{V_C})
-e^{\frac{1}{24}A_3}A_3\widehat{A}(TX){\rm exp}(\frac{c}{2})\right\}^{(10)}.
\notag
\end{align}
\end{thm}
\begin{proof}
By Lemmas 3.1 and 3.2, we have
\begin{align}
&{Q_2}(M,P_i,\tau)=h_0(8\delta_2)^{4}+h_1(8\delta_2)^{2}\varepsilon_2+h_2\varepsilon_2^2,
\end{align}
\begin{align}
&{Q_1}(M,P_i,\tau)=2^{\overline{l}}[h_0(8\delta_1)^{4}+h_1(8\delta_1)^{2}\varepsilon_1+h_2\varepsilon_1^2],
\end{align}
where
each $h_r,~ 0\leq r\leq 2,$ is a real multiple of the
volume form at $x$.
\begin{align}
&{Q_2}(M,P_i,\tau)\\\notag
&=(e^{\frac{1}{24}A_3}-e^{\frac{1}{24}A_3}A_3q+O(q^2))\widehat{A}(TM){\rm exp}(\frac{c}{2})
{\rm ch}(1+q(\widetilde{T_CM}-\widetilde{L_R})+O(q^2))\\\notag
&\cdot{\rm ch}(1-q^{\frac{1}{2}}\widetilde{V_C}+q\wedge^2\widetilde{V_C}-q^{\frac{3}{2}}
(\wedge^3\widetilde{V_C}+\widetilde{V_C})+O(q^2))
(1-8q+O(q^2))(1+{\rm ch}({W_i})q+O(q^2))\\\notag
&=e^{\frac{1}{24}A_3}\widehat{A}(TM){\rm exp}(\frac{c}{2})-q^{\frac{1}{2}}e^{\frac{1}{24}A_3}\widehat{A}(TX){\rm exp}(\frac{c}{2}){\rm ch}(\widetilde{V_C})\\\notag
&+q[e^{\frac{1}{24}A_3}\widehat{A}(TM){\rm exp}(\frac{c}{2}){\rm ch}(\widetilde{T_CM}-\widetilde{L_R}-8+\wedge^2\widetilde{V_C}+W_i)-e^{\frac{1}{24}A_3}A_3\widehat{A}(TM){\rm exp}(\frac{c}{2})]
+O(q^{\frac{3}{2}}).\notag
\end{align}
Comparing the the constant term, $q^{\frac{1}{2}},q$ terms in (3.9), we get
\begin{align}
&h_0=\left\{e^{\frac{1}{24}A_3}\widehat{A}(TM){\rm exp}(\frac{c}{2}))\right\}^{(10)}\\\notag
&h_1+96h_0=-\left\{e^{\frac{1}{24}A_3}\widehat{A}(TX){\rm exp}(\frac{c}{2}){\rm ch}(\widetilde{V_C})\right\}^{(10)}\\\notag
&h_2+56h_1+3552h_0=\left\{e^{\frac{1}{24}A_3}\widehat{A}(TM){\rm exp}(\frac{c}{2}){\rm ch}(\widetilde{T_CM}-\widetilde{L_R}-8+\wedge^2\widetilde{V_C}+W_i)\right.\\\notag
&\left.-e^{\frac{1}{24}A_3}A_3\widehat{A}(TM){\rm exp}(\frac{c}{2})]\right\}^{(10)}\\\notag
\end{align}
By (3.12), (3.10) and (3.5),  comparing the the constant term in (3.10), we get Theorem 3.3.
\end{proof}
\begin{cor} We have the following equality when $A_3=0$:
\begin{align}
&\left\{\widehat{A}(TX){\rm exp}(\frac{c}{2}){\rm ch}(\triangle(V))\right\}^{(10)}=
2^{\overline{l}-8}\left\{\widehat{A}(TM){\rm exp}(\frac{c}{2})\right.\\\notag
&\left.\cdot{\rm ch}[-252-8\widetilde{V_C}+\widetilde{T_CM}-\widetilde{L_R}+W_i+\wedge^2\widetilde{V_C})
\right\}^{(10)}.
\notag
\end{align}
When $M$ is spin$^c$ and $10$-dimensional and $\overline{l}>8$, then ${\rm Ind}(D_c\otimes \triangle(V))_+$ is a multiple of $2^{\overline{l}-8}$.
\end{cor}
Let $M$ be a $14$-dimensional spin$^c$ manifold and we define $\widetilde{Q_j}(P_i,V,\tau)$ for $1\leq j\leq 3$ using the same expressions (3.1)-(3.3) of ${Q_j}(P_i,V,\tau)$.
We have
\begin{lem}
$\widetilde{Q_1}(M,P_i,\tau)$ is a
modular form of weight $10$ over $\Gamma_0(2)$, while $\widetilde{Q_2}(M,P_i,\tau)$ is
a modular form of weight $10$ over $\Gamma^0(2)$ . Moreover, the
following identity holds,
\begin{align}
\widetilde{Q_1}(M,P_i,-\frac{1}{\tau})=2^{\overline{l}}\tau^{10}\widetilde{Q_2}(M,P_i,\tau).
\end{align}
\end{lem}

By Lemmas 3.1 and 3.5, we have
\begin{align}
&\widetilde{Q_2}(M,P_i,\tau)=\widetilde{h_0}(8\delta_2)^{5}+\widetilde{h_1}(8\delta_2)^{3}\varepsilon_2+\widetilde{h_2}(8\delta_2)\varepsilon_2^2,
\end{align}
\begin{align}
&\widetilde{Q_1}(M,P_i,\tau)=2^{\overline{l}}[\widetilde{h_0}(8\delta_1)^{5}
+\widetilde{h_1}(8\delta_1)^{3}\varepsilon_1+\widetilde{h_2}(8\delta_1)\varepsilon_1^2],
\end{align}
where
each $\widetilde{h_r},~ 0\leq r\leq 2,$ is a real multiple of the
volume form at $x$.
Using the same tricks in Theorem 3.3, we have
\begin{thm} We have the following equality:
\begin{align}
&\left\{e^{\frac{1}{24}A_3}\widehat{A}(TX){\rm exp}(\frac{c}{2}){\rm ch}(\triangle(V))\right\}^{(14)}=
2^{\overline{l}-7}\left\{e^{\frac{1}{24}A_3}\widehat{A}(TM){\rm exp}(\frac{c}{2})\right.\\\notag
&\left.\cdot{\rm ch}(144-16\widetilde{V_C}-\widetilde{T_CM}+\widetilde{L_R}-W_i-\wedge^2\widetilde{V_C})
+e^{\frac{1}{24}A_3}A_3\widehat{A}(TX){\rm exp}(\frac{c}{2})\right\}^{(14)}.
\notag
\end{align}
\end{thm}
\begin{cor} We have the following equality when $A_3=0$:
\begin{align}
&\left\{\widehat{A}(TX){\rm exp}(\frac{c}{2}){\rm ch}(\triangle(V))\right\}^{(14)}=
2^{\overline{l}-7}\left\{\widehat{A}(TM){\rm exp}(\frac{c}{2})\right.\\\notag
&\left.\cdot{\rm ch}(144-16\widetilde{V_C}-\widetilde{T_CM}+\widetilde{L_R}-W_i-\wedge^2\widetilde{V_C})
\right\}^{(14)}.
\notag
\end{align}
When $M$ is spin$^c$ and $14$-dimensional and $\overline{l}>7$, then ${\rm Ind}(D_c\otimes \triangle(V))_+$ is a multiple of $2^{\overline{l}-7}$.
\end{cor}
Let $M$ be a $14$-dimensional spin$^c$ manifold and let
\begin{align}
Q_{\overline{j}}(M,P_i,P_j,V,\tau)=&\left\{e^{\frac{1}{24}E_2(\tau)A_2}\widehat{A}(TX){\rm exp}(\frac{c}{2}){\rm ch}\left[\bigotimes _{n=1}^{\infty}S_{q^n}(\widetilde{T_CM})
\otimes
\bigotimes _{m=1}^{\infty}\wedge_{-q^m}(\widetilde{L_C})\right]\right.\\\notag
&\left.\cdot{\rm ch}(Q_{\overline{j}}(V))\varphi(\tau)^{16}{\rm ch}(\mathcal{V}_i){\rm ch}(\mathcal{V}_j)\right\}^{(14)},~~1\leq \overline{j} \leq 3.
\end{align}
Then
\begin{align}Q_1(M,P_i,P_j,V,\tau)=&\left\{e^{\frac{1}{24}E_2(\tau)A_2}\left(\prod_{j=1}^{7}\frac{x_j\theta'(0,\tau)}{\theta(x_j,\tau)}\right)
\frac{\sqrt{-1}\theta(u,\tau)}{\theta_1(0,\tau)\theta_2(0,\tau)
\theta_3(0,\tau)}\right.\\\notag
&2^{\overline{l}}\cdot\prod_{\alpha=1}^{\overline{l}}\frac{\theta_1(\overline{y_{\alpha}},\tau)
}{\theta_1(0,\tau)}\frac{1}{4}\left(\prod_{l=1}^8\theta_1(y_l^i,\tau)+\prod_{l=1}^8\theta_2(y_l^i,\tau)
+\prod_{l=1}^8\theta_3(y_l^i,\tau)\right)\\\notag
&\left.\left(\prod_{l=1}^8\theta_1(y_l^j,\tau)+\prod_{l=1}^8\theta_2(y_l^j,\tau)+\prod_{l=1}^8\theta_3(y_l^j,\tau)\right)
\right\}^{(14)},
\end{align}
\begin{align}Q_{\overline{j}}(M,P_i,V,\tau)=&\left\{e^{\frac{1}{24}E_2(\tau)A_2}\left(\prod_{j=1}^{7}\frac{x_j\theta'(0,\tau)}{\theta(x_j,\tau)}\right)
\frac{\sqrt{-1}\theta(u,\tau)}{\theta_1(0,\tau)\theta_2(0,\tau)
\theta_3(0,\tau)}\right.\\\notag
&\cdot\prod_{\alpha=1}^{\overline{l}}\frac{\theta_{\overline{j}}(\overline{y_{\alpha}},\tau)
}{\theta_{\overline{j}}(0,\tau)}\frac{1}{4}\left(\prod_{l=1}^8\theta_1(y_l^i,\tau)+
\prod_{l=1}^8\theta_2(y_l^i,\tau)+\prod_{l=1}^8\theta_3(y_l^i,\tau)\right)\\\notag
&\left.\left(\prod_{l=1}^8\theta_1(y_l^j,\tau)+\prod_{l=1}^8\theta_2(y_l^j,\tau)+\prod_{l=1}^8\theta_3(y_l^j,\tau)\right)\right\}^{(14)},
\end{align}
where $2\leq {\overline{j}} \leq 3$.
\begin{lem}
${Q_1}(M,P_i,P_j,\tau)$ is a
modular form of weight $14$ over $\Gamma_0(2)$, while ${Q_2}(M,P_i,P_j,\tau)$ is
a modular form of weight $14$ over $\Gamma^0(2)$ . Moreover, the
following identity holds,
\begin{align}
{Q_1}(M,P_i,P_j,-\frac{1}{\tau})=2^{\overline{l}}\tau^{14}{Q_2}(M,P_i,P_j,\tau).
\end{align}
\end{lem}
By Lemmas 3.1 and 3.8, we have
\begin{align}
&{Q_2}(M,P_i,P_j,\tau)=h_0'(8\delta_2)^{7}+h_1'(8\delta_2)^{5}\varepsilon_2+h_2'(8\delta_2)^{3}\varepsilon_2^2+h_3'(8\delta_2)\varepsilon_2^3,
\end{align}
\begin{align}
&{Q_1}(M,P_i,P_j,\tau)=2^{\overline{l}}[h_0'(8\delta_1)^{7}+h_1'(8\delta_1)^{5}\varepsilon_1+h_2'(8\delta_1)^{3}\varepsilon_1^2+h_3'(8\delta_1)\varepsilon_1^3],
\end{align}
where
each $h_r',~ 0\leq r\leq 3,$ is a real multiple of the
volume form at $x$.
\begin{align}
&{Q_2}(M,P_i,P_j\tau)\\\notag
&=(e^{\frac{1}{24}A_2}-e^{\frac{1}{24}A_2}A_2q+O(q^2))\widehat{A}(TM){\rm exp}(\frac{c}{2})
{\rm ch}(1+q(\widetilde{T_CM}-\widetilde{L_R})+O(q^2))\\\notag
&\cdot{\rm ch}(1-q^{\frac{1}{2}}\widetilde{V_C}+q\wedge^2\widetilde{V_C}-q^{\frac{3}{2}}
(\wedge^3\widetilde{V_C}+\widetilde{V_C})+O(q^2))\\\notag
&(1-16q+O(q^2))(1+{\rm ch}({W_i}+W_j)q+O(q^2))\\\notag
&=e^{\frac{1}{24}A_2}\widehat{A}(TM){\rm exp}(\frac{c}{2})-q^{\frac{1}{2}}e^{\frac{1}{24}A_2}\widehat{A}(TX){\rm exp}(\frac{c}{2}){\rm ch}(\widetilde{V_C})\\\notag
&+q[e^{\frac{1}{24}A_2}\widehat{A}(TM){\rm exp}(\frac{c}{2}){\rm ch}(\widetilde{T_CM}-\widetilde{L_R}-16+\wedge^2\widetilde{V_C}+W_i+W_j)-e^{\frac{1}{24}A_2}A_2\widehat{A}(TM){\rm exp}(\frac{c}{2})]\\\notag
&+q^{\frac{3}{2}}[-e^{\frac{1}{24}A_2}\widehat{A}(TM){\rm exp}(\frac{c}{2}){\rm ch}
(\wedge^3\widetilde{V_C}+\widetilde{V_C}+(\widetilde{T_CM}-\widetilde{L_R}-16+W_i+W_j)\otimes\widetilde{V_C})\\\notag
&+e^{\frac{1}{24}A_2}A_2\widehat{A}(TX){\rm exp}(\frac{c}{2}){\rm ch}(\widetilde{V_C})]+O(q^2).\notag
\end{align}
Comparing the the constant term, $q^{\frac{1}{2}},q,~q^{\frac{3}{2}}$ terms in (3.23) and using the same tricks in Theorem 3.3, we get
\begin{thm} We have the following equality when $A_2=0$:
\begin{align}
&\left\{\widehat{A}(TX){\rm exp}(\frac{c}{2}){\rm ch}(\triangle(V))\right\}^{(14)}=
2^{\overline{l}-11}\left\{\widehat{A}(TM){\rm exp}(\frac{c}{2})\right.\\\notag
&\cdot{\rm ch}(-3008+301\widetilde{V_C}+\wedge^3\widetilde{V_C}+24\widetilde{T_CM}-24\widetilde{L_R}+24W_i+24W_j+24\wedge^2\widetilde{V_C}\\\notag
&\left.+(\widetilde{T_CM}-\widetilde{L_R}-16+W_i+W_j)\otimes\widetilde{V_C})\right\}^{(14)}.
\notag
\end{align}
When $M$ is spin$^c$ and $14$-dimensional and $\overline{l}>11$, then ${\rm Ind}(D_c\otimes \triangle(V))_+$ is a multiple of $2^{\overline{l}-11}$.
\end{thm}
Let $M$ be a $10$-dimensional spin$^c$ manifold and we define $\widetilde{Q_{\overline{j}}}(P_i,P_j,V,\tau)$ for $1\leq \overline{j}\leq 3$ using the same expressions (3.19)-(3.21) of ${Q_{\overline{j}}}(P_i,P_j,V,\tau)$.
We have
\begin{lem}
$\widetilde{Q_1}(M,P_i,P_j,\tau)$ is a
modular form of weight $12$ over $\Gamma_0(2)$, while $\widetilde{Q_2}(M,P_i,P_j,\tau)$ is
a modular form of weight $12$ over $\Gamma^0(2)$ . Moreover, the
following identity holds,
\begin{align}
\widetilde{Q_1}(M,P_i,P_j,-\frac{1}{\tau})=2^{\overline{l}}\tau^{12}\widetilde{Q_2}(M,P_i,P_j,\tau).
\end{align}
\end{lem}
By Lemmas 3.1 and 3.10, we have
\begin{align}
&\widetilde{Q_2}(M,P_i,P_j,\tau)=h_0"(8\delta_2)^{6}+h_1"(8\delta_2)^{4}\varepsilon_2+h_2"(8\delta_2)^{2}\varepsilon_2^2+h_3"\varepsilon_2^3,
\end{align}
\begin{align}
&\widetilde{Q_1}(M,P_i,P_j,\tau)=2^{\overline{l}}[h_0"(8\delta_1)^{6}+h_1"(8\delta_1)^{4}\varepsilon_1
+h_2"(8\delta_1)^{2}\varepsilon_1^2+h_3"\varepsilon_1^3],
\end{align}
where
each $h_r",~ 0\leq r\leq 3,$ is a real multiple of the
volume form at $x$. Similar to the theorem 3.9, we have
\begin{thm} We have the following equality when $A_2=0$:
\begin{align}
&\left\{\widehat{A}(TX){\rm exp}(\frac{c}{2}){\rm ch}(\triangle(V))\right\}^{(10)}=
2^{\overline{l}-12}\left\{\widehat{A}(TM){\rm exp}(\frac{c}{2})\right.\\\notag
&\left.\cdot{\rm ch}(4096+25\widetilde{V_C}-\wedge^3\widetilde{V_C}+(\widetilde{T_CM}-\widetilde{L_R}-16+W_i+W_j)\otimes\widetilde{V_C})\right\}^{(10)}.\notag
\end{align}
When $M$ is spin$^c$ and $10$-dimensional and $\overline{l}>12$, then ${\rm Ind}(D_c\otimes \triangle(V))_+$ is a multiple of $2^{\overline{l}-12}$.
\end{thm}

\section{Acknowledgements}
The author was supported by Science and Technology Development Plan Project of Jilin Province, China: No.20260102245JC.

\vskip 1 true cm


\bigskip
\bigskip

\indent{Yong Wang, School of Mathematics and Statistics,
Northeast Normal University, Changchun Jilin, 130024, China }\\
\indent E-mail: {\it wangy581@nenu.edu.cn }\\
\end{document}